\documentclass[a4paper]{article}

\usepackage[]{geometry}

\usepackage[T1]{fontenc}
\usepackage[french,english]{babel}

\usepackage{amsmath,mathtools}
\usepackage{amsthm,thmtools}
\usepackage{etoolbox}

\usepackage{libertine}
\usepackage[libertine]{newtxmath}
\usepackage[mono,narrow,scale=.95]{zi4}
\usepackage[frak=euler,scr=euler,bb=libus]{mathalpha}

\usepackage{microtype}
\usepackage{indentfirst}
\usepackage[all]{nowidow}

\usepackage[shortlabels]{enumitem}
\usepackage{booktabs}
\usepackage{csquotes}
\usepackage{graphicx}
\usepackage{tikz}
\usepackage{tikz-cd}
\usepackage[all]{xy}
\usepackage{pdfpages}
\usepackage{appendix}
\usepackage{varwidth}

\usepackage[makeindex]{imakeidx}
\makeindex[options=-s index.ist]

\usepackage[unicode]{hyperref}
\hypersetup{
    colorlinks=true,
    urlcolor=blue,
    linkcolor=blue,
    citecolor=teal,
    breaklinks=true,
    hypertexnames=false
}
\usepackage[nameinlink]{cleveref}

\usepackage[backend=biber]{biblatex}
\AtEveryBibitem{\clearfield{issn}}

\usepackage{authblk}
\makeatletter
\def\authprecit#1{
  \@temptokena=\expandafter{\AB@authlist}\xdef\AB@authlist{\the\@temptokena
  \noexpand\\{\noexpand\normalsize\noexpand\itshape#1}
  \noexpand\\\noexpand\hspace{-2em}}}
\makeatother

\usepackage{fancyhdr}
\makeatletter
\fancypagestyle{plain}{
  \fancyhf{}
  
}
\fancypagestyle{courant}{
  \fancyhf{}
  
  \fancyhead[L]{\small\em\@title}
  \fancyhead[R]{\small\thepage}
}
\makeatother
\makeatletter
\long\def\nonumfootnote#1{%
  \begingroup
  \renewcommand\thefootnote{}%
  \renewcommand\@makefntext[1]{\noindent##1}%
  \footnotetext{#1}%
  \endgroup}
\makeatother

\newcommand{\subjclass}[2][2020]{%
  \nonumfootnote{\textit{#1 Mathematics Subject Classification.}~#2.}}

\newcommand{\keywords}[1]{%
  \nonumfootnote{\textit{Keywords and phrases:}~#1.}}

\newcommand{\support}[1]{\nonumfootnote{#1}}

\AtEndEnvironment{verbatim}{\pagebreak[3]}

\makeatletter
\g@addto@macro\bfseries{\boldmath}
\g@addto@macro\mdseries{\unboldmath}
\g@addto@macro\normalfont{\unboldmath}
\makeatother

\usepackage{pdfrender}
\def\pdfbold#1{%
  \textpdfrender{TextRenderingMode=FillStroke,%
  LineJoinStyle=Round,LineWidth=\baselineskip/90}{#1}}

\makeatletter
\def\string@b{b}
\let\oldmathbb\mathbb
\def\mathbfbb#1{\pdfbold{\boldsymbol{\oldmathbb{#1}}}}
\def\mathbb#1{{\texorpdfstring{\ifx\f@series\string@b\mathbfbb{#1}\else\oldmathbb{#1}\fi}{#1}}}
\makeatother

\let\ge\geqslant
\let\le\leqslant
\let\geq\geqslant
\let\leq\leqslant

\let\epsilon\varepsilon

\makeatletter
\def\thmhead#1#2#3{%
  \thmname{#1}\thmnumber{\@ifnotempty{#1}{ }\@upn{#2}}%
  \thmnote{ {\normalfont(#3)}}}
\makeatother

\declaretheorem[name=Theorem,numberwithin=section]{thm}
\declaretheorem[name=Definition,sibling=thm]{defi}
\declaretheorem[name=Proposition,sibling=thm]{prop}
\declaretheorem[name=Lemma,sibling=thm]{lem}
\declaretheorem[name=Corollary,sibling=thm]{cor}
\declaretheorem[name=Algorithm,sibling=thm]{alg}

\declaretheorem[name=Remark,sibling=thm,style=remark]{rem}

\declaretheorem[name=Theorem,style=plain,unnumbered]{thm*}
\declaretheorem[name=Definition,style=plain,unnumbered]{defi*}
\declaretheorem[name=Proposition,style=plain,unnumbered]{prop*}
\declaretheorem[name=Lemma,style=plain,unnumbered]{lem*}
\declaretheorem[name=Corollary,style=plain,unnumbered]{cor*}
\declaretheorem[name=Algorithm,style=plain,unnumbered]{alg*}
\declaretheorem[name=Claim,style=plain,unnumbered]{claim*}
\declaretheorem[name=Conjecture,style=plain,unnumbered]{conjecture*}
\declaretheorem[name=Hypothesis,style=plain,unnumbered]{hypothesis*}
\declaretheorem[name=Problem,style=plain,unnumbered]{problem*}
\declaretheorem[name=Question,style=plain,unnumbered]{question*}
\declaretheorem[name=Remark,style=remark,unnumbered]{rem*}
\declaretheorem[name=Example,style=remark,unnumbered]{ex*}
\declaretheorem[name=Notation,style=remark,unnumbered]{notation*}

\crefname{thm}{theorem}{theorems}
\Crefname{thm}{Theorem}{Theorems}
\crefname{defi}{definition}{definitions}
\Crefname{defi}{Definition}{Definitions}
\crefname{prop}{proposition}{propositions}
\Crefname{prop}{Proposition}{Propositions}
\crefname{lem}{lemma}{lemmas}
\Crefname{lem}{Lemma}{Lemmas}
\crefname{cor}{corollary}{corollaries}
\Crefname{cor}{Corollary}{Corollaries}
\crefname{alg}{algorithm}{algorithms}
\Crefname{alg}{Algorithm}{Algorithms}
\crefname{claim}{claim}{claims}
\Crefname{claim}{Claim}{Claims}
\crefname{conjecture}{conjecture}{conjectures}
\Crefname{conjecture}{Conjecture}{Conjectures}
\crefname{hypothesis}{hypothesis}{hypotheses}
\Crefname{hypothesis}{Hypothesis}{Hypotheses}
\crefname{problem}{problem}{problems}
\Crefname{problem}{Problem}{Problems}
\crefname{question}{question}{questions}
\Crefname{question}{Question}{Questions}
\crefname{rem}{remark}{remarks}
\Crefname{rem}{Remark}{Remarks}
\crefname{ex}{example}{examples}
\Crefname{ex}{Example}{Examples}
\crefname{notation}{notation}{notations}
\Crefname{notation}{Notation}{Notations}

\DeclareMathOperator{\Trace}{Tr}

\DeclareMathOperator{\Gal}{Gal}

\DeclareMathOperator{\GL}{GL}
\DeclareMathOperator{\End}{End}
\DeclareMathOperator{\Ker}{Ker}

\DeclareMathOperator{\Aut}{Aut}

\DeclareMathOperator{\Spec}{Spec}

\DeclareMathOperator{\Char}{char}

\newcommand{\F}{\mathbb{F}}

\newcommand{\Q}{\mathbb{Q}}
\newcommand{\Z}{\mathbb{Z}}

\newcommand{\p}{\mathfrak{p}}
\newcommand{\h}{\mathfrak{h}}
\newcommand{\OO}{\mathcal{O}}

\def\fonction#1#2#3#4{
\left\{
\begin{array}{@{}r@{~}c@{~}l@{}}
#1 & \longrightarrow & #2 \\
#3 & \longmapsto & #4
\end{array}
\right.
}

\title{Automorphisms of finite fields from isogeny cycles}
\author{Kéva Djambaé}
\affil{Laboratoire GAATI, Université de Polynésie Française}
\affil{Institut de Mathématiques de Marseille (I2M), UMR 7373 CNRS, Aix-Marseille Université}

\date{}

\begin{document}
\maketitle

\begin{abstract}
We develop an explicit geometric construction of automorphisms of finite fields arising from isogeny cycles. 
Let $k$ be a finite field, $E/k$ an elliptic curve, and $\ell$ an integer coprime to $\Char(k)$. 
Let $\p \subseteq \End(E)$ be an ideal dividing $(\ell)$, and consider the corresponding torsion subgroup $E[\p] \subseteq E[\ell]$. 
From the action of $\End(E)$ on $E[\p]$, we construct the splitting field $K$ of the $x$-coordinates of points in $E[\p]$ and the associated Galois group $\Gal(K/k)$. This yields  $(\End(E)/\p)^\ast \to \Gal(K/k)$ a group homomorphism.
\end{abstract}

\support{The author is supported by a PhD grant from the University of French Polynesia and by the ANR MELODIA Project under grant number ANR-20-CE40-0013.}




\section{Introduction}
Let $K/k$ be an extension of finite fields, denote by $q$ the cardinality of the field $k$. The Galois group $\Gal(K/k)$ is  cyclic of order $[K:k]$, generated by the Frobenius automorphism $\pi_q$, and such extensions play a crucial role in cryptography, notably in pairing-based protocols \cite{Boneh}, which remain an active area of research today \cite{Aranha}.

\smallskip

In this work, we investigate how isogenies between elliptic curves can yield explicit constructions of field automorphisms. Our approach builds on Kohel’s study of endomorphism rings \cite{Kohel1996} and connects to recent work by Bassa, Bisson, and Oyono on iterative constructions of irreducible polynomials via isogenies \cite{Bassa2024}, as well as to \cite{BissonTibouchi2018}, which uses isogenies of elliptic curves to construct permutation rational functions over finite fields.

\smallskip

Let $k$ be a finite field and $E/k$ an elliptic curve. We exploit the scheme-theoretic structure of the $\ell$-torsion subgroup $E[\ell]$ to construct explicit Galois automorphisms of the field extensions generated by the $x$-coordinates of the $\ell$-torsion points. Following \cite{cryptoeprint:2025/543}, we restrict the action of $\End(E)$ on the Kummer line $\mathbb{P}^1(K)$ to the subgroup $E[\p](K)$, for a prime ideal $\p$ dividing $(\ell)$, and consider a Galois orbit $X \subset E[\p](K)$ on the Kummer line. This procedure allows us to generate elements of the Galois group independently of the Frobenius, with no knowledge of the associated discrete logarithm, that is, without a priori knowing the exponent $i$ such that $\sigma_f = \pi_q^i$ in the cyclic group $\Gal(K/k) = \langle \pi_q \rangle$.
\smallskip

This article is organized as follows. In Section \ref{sec:prelim}, we review the necessary background, recalling the key ideas
and structural arguments underlying construction. Section \ref{sec:construction}  develops the  construction in case where $\ell$ is a prime different from $p$, the characteristic of $k$. Section \ref{General case} generalizes the construction to all positive integers $\ell$ and proves that suitable endomorphisms yield explicit Galois automorphisms. Section \ref{sec:algo} begins by proposing to construct an endomorphism $f$ of an elliptic curve $E$ using isogeny graphs, then discusses the algorithm and aspects related to the complexity of the construction. Section \ref{sec:applications}  illustrates the method through detailed examples
in both the ordinary and supersingular settings.

\smallskip

The primary goal of this paper is to establish the following result, which relates to the explicit construction of Galois
automorphisms :

\begin{thm}[Main result]\label{thm::final}
Let $k = \F_{q} $ be a finite field of characteristic $p$ and let $\ell$ be a positive integer coprime to $p$.
Let $E/k$ be an elliptic curve, and let $\p$ be an ideal  of $\End(E)$ of  norm $\ell$.
Denote by $H$  the subscheme of $E[\ell]$ defined by
$$
H=E[\p]=\bigcap_{\phi\in\p} \Ker(\phi)\subseteq E[\ell] .
$$
Denote by $\psi$ an irreducible factor of the kernel polynomial $\psi_H$ of $H$ and by $n$ the degree of $\psi$ . Let $K$ be the splitting field of $\psi$ over $k$, and suppose that the roots $(x_i)_{1 \le i \le n}$ of $\psi$ form a basis of $K/k$.
If $f \in \End(E)$ is a non trivial endomorphism defined over $k$ whose degree is coprime to $\ell$, then $f$ induces an automorphism
$$
\sigma_f \in \Gal(K/k), \quad \text{defined by} \quad \sigma_f\left( \sum_{i=1}^{n} \alpha_i x_i \right) = \sum_{i=1}^{n} \alpha_i \bar{f}(x_i),
$$
for any $\alpha_i \in K$, where $\bar{f}$ is the induced endomorphism of the Kummer curve $\mathbb{P}^1(k) = E/\{\pm1\}$.
\end{thm}


\section{Background}\label{sec:prelim}
Let $k=\F_q$ be a finite field. Denote by $\overline{k}$  its algebraic closure. Its Galois group $\Gal(K/k)$ is cyclic and generated by the Frobenius endomorphism $\pi_q$.

\subsection{Elliptic curves and torsion points}
Let $E/k$ be an elliptic curve. Recall that $E$ is a smooth group scheme of finite type over $k$.
For any integer $n\ge 1$ with $\gcd(n,\mathrm{char}(k))=1$, the subgroup scheme $E[n]$ is finite étale of rank $n^2$.
For a prime $\ell\neq \Char(k)$, we denote by $\psi_\ell$ the $\ell$-division polynomial, whose roots are the
$x$-coordinates of non-trivial $\ell$-torsion points of $E$.

\begin{prop}
\label{prop:splitting_field}
Let $\psi$ be an irreducible factor of the division polynomial $\psi_\ell$ over $k$ and put $ K = k[X]/(\psi(X)). $ Then $K/k$ is
a finite separable extension generated over $k$ by the  $x$-coordinate of a non-trivial $\ell$-torsion point.
\end{prop}

\begin{proof}
Since $\gcd(\ell,\mathrm{char}(k))=1$, the subgroup scheme $E[\ell]$ is finite étale over $k$, so $\psi_\ell$ has simple roots and every irreducible factor $\psi$ is separable, giving that $K/k$ is a finite separable extension. 

By \cite[Chapter~III, Exercise~3.7 (a),(f)]{Silverman1986}, the roots of $\psi_\ell$ are exactly the $x$-coordinates of the non-trivial $\ell$-torsion points of $E$, so $K=k[X]/(\psi(X))$ is generated over $k$ by the $x$-coordinate of such a point.
\end{proof}

The structure of the endomorphism ring $\End(E)$ of an elliptic curve $E$ depends on whether $E$ is ordinary or supersingular, as follows:

\def\stuff{\cite[Chapter~V, Theorem~3.1]{Silverman1986}}
\begin{prop}[\stuff]
Suppose $k$ is a field of finite characteristic, if $E$ is ordinary, then $\End(E)$ is an order in an imaginary quadratic field. Otherwise, $E$ is supersingular and $\End(E)$ is a maximal order in a quaternion algebra.
\end{prop}

\subsection{Coordinate Rings and Functoriality}
Let $H\subset E$ be  a finite subgroup scheme. We denote by $R_H$ the coordinate ring of $H$. When $H$ is
étale, $R_H$ is a finite $k$-algebra. We have $$ R_H\simeq\prod_{Z\subset H} k(Z) $$ where each $k(Z)$ is the
residue field at the closed point $Z\subset H$. Under this condition, each $Z$ is a Galois orbit of some
point $P\in H(\overline{k})$.
We have a functorial anti-equivalence (Cartier’s correspondence) between finite $k$-group schemes and finite
$k$-Hopf algebras.

\begin{rem}
This correspondence allows one to translate the action of group-scheme endomorphisms on points into an action
on the coordinate ring.
\end{rem}
Hence, any group-scheme endomorphism $f_H \colon H \to H$ induces by pullback a $k$-algebra endomorphism
$f_H^* \colon R_H \to R_H$, and conversely (see \cite{SGA3} for scheme-theoretic duality and
functoriality).

\subsection{Normal basis}
For more detail around links between normal basis and elliptic curves, we refer to \cite{CouveignesLercier2009}. To begin, recall that every finite extension of finite field is Galois.

\begin{defi}
A basis of $K$ over $k$ is normal if it is of the form $\{\sigma(\alpha)\mid\sigma \in \Gal(K/k)\}$ for some $\alpha \in K$.
\end{defi}

\begin{prop}
Let $k=\F_q$ be a finite field of characteristic $p$ and $\psi\in k[X]$ an irreducible polynomial of degree $n$ with $\gcd(n,p)=1$ and $K=k[X]/(\psi(X))$ the decomposition field of $\psi$. Denote by $(x_i)_{1\leq i\leq n}$ the roots of $\psi$. Then $(x_i)_{1\leq i\leq n}$ is a $k$-basis if and only if $\Trace_{K/k}(x_1)\neq 0$.
\end{prop}

\begin{proof}
Since $\psi$ is irreducible of degree $n$, we have $[K:k]=n$, the extension $K/k$ is cyclic and
$G=\Gal(K/k)=\langle\pi_q\rangle,$ where $\pi_q\colon K\to K$ is the Frobenius automorphism. Set $x = x_1$, and order $(x_i)$ such that $\pi_q(x_i) = x_{i+1}.$

\smallskip

If $Tr(x) = \sum_i x_i = 0$, then $(x_i)_i$ are linearly dependent, so we may assume $Tr(x) \ne 0$. Via the isomorphism
$$
\iota\colon\fonction{k[G]}{k[X]/(X^n-1)}{\pi_q}{X},
$$
we identify $k[G]$ with $R := k[X]/(X^n-1)$. Consider the $k$-linear map
$$
\alpha\colon\fonction{R = k[X]/(X^n-1)}{K}{P(X)}{P(\pi_q)(x)}.
$$
Its image is the $k$-span of $(\pi_q^{\,i}(x))_{0\leq i\leq n-1}$, that is the $k$-span of $(x_i)_{1\leq i\leq n}$. Since $k[G]\simeq k[X]/(X^n-1)$ and $\dim_k k[G]=n=\dim_k K$, it suffices to show that $\Ker(\alpha)=0$.

\smallskip

Since $\gcd(n,p)=1$, the polynomial $N(X):=1+X+\cdots+X^{n-1}$ satisfies $N(1)=n\neq0$ in $k$, so $X-1$ and $N(X)$ are coprime in $k[X]$. By the Chinese Remainder Theorem, $R\cong k\times K'$, where $K'=k[X]/(N(X))$, and under this isomorphism $(N(X))$ corresponds to the ideal $k\times0$ while $(X-1)$ corresponds to $0\times K'$; these are the only nonzero proper ideals of $R$.

\smallskip

Suppose $\Ker(\alpha)\neq0$. Since $\alpha(N)=(1+\pi_q+\cdots+\pi_q^{n-1})(x)=\sum_{i=0}^{n-1}\pi_q^{\,i}(x)=\Trace_{K/k}(x)\neq0$ by assumption, $N\notin\Ker(\alpha)$, so $\Ker(\alpha)\neq k\times0$ and $\Ker(\alpha)\neq R$. If $\Ker(\alpha)=0\times K'$, then $X-1\in\Ker(\alpha)$, i.e. $\alpha(X-1)=\pi_q(x)-x=0$, so $x\in k$, contradicting $\deg\psi=n\geq2$. Hence $\Ker(\alpha)=0$. Therefore $\alpha$ is an isomorphism and $(x_i)_{1\leq i\leq n}$ is a $k$-basis of $K$.
\end{proof}

\begin{rem}
According to \cite{HuangHanCao2018,lenstra-schoof}, let $n=(\ell-1)/2$ be a prime number distinct from $p$, let $q$ be primitive modulo $n$, and let $\psi$ be an irreducible polynomial of degree $n$ with nonzero $X^{n-1}$-coefficient. Then $\psi$ is an $N$-polynomial, and its roots $(x_i)_{1\leq i\leq n}$ form a basis of the extension $K/k$.
\end{rem}


\section{Constructing automorphisms of finite fields from elliptic curves}\label{sec:construction}
In this section, we construct explicit  automorphisms of extensions generated by $\ell$-torsion coordinates,
using endomorphisms of the elliptic curve $E/k$. The key tool is the $\ell$-division polynomial $\psi_\ell$, whose
roots encode the $x$-coordinates of non-trivial $\ell$-torsion points.

\subsection{Case when $\ell$ is a prime}
Let $k$ be a finite field, and let $E/k$ be an elliptic curve. Fix an odd prime $\ell\neq\Char(k)$. Let $\psi_\ell$ denote the $\ell$-division polynomial of $E$.

Since $E[\ell] \cong (\mathbb{Z}/\ell\mathbb{Z})^2$,  there are $\frac{\ell^2-1}{2}$ nonzero points of order $\ell$ modulo the elliptic involution, which come in
pairs $\{P,-P\}$ with identical $x$-coordinates. For any ideal $\p$ of $\End(E)$ dividing $(\ell)$, we define the subscheme
$$
E[\p]=\bigcap_{\phi\in\p} \Ker(\phi)\simeq\Z/\ell\Z.
$$

\begin{lem}
Let $E/k$ be an elliptic curve over the finite field $k=\F_q$ and let $\ell\neq\mathrm{char}(k)$ be a prime.
\begin{enumerate}
\item If $E$ is ordinary and $\p\subset\End(E)$ is any  ideal dividing $(\ell)$, then
$$
E[\p]=\{P\in E[\ell]\mid \forall\alpha\in\p,  \alpha(P)=0 \}
$$
is a cyclic subgroup of order $\ell$, hence $E[\p]\subset E[\ell]$.
\item If $E$ is supersingular and $\p\subset\End(E)$ is a two-sided ideal of reduced norm $\ell$, then there exists an isogeny $\varphi_{\p}\colon E\to E/E[\p]$ corresponding to $\p$, whose kernel $E[\p]$ is a cyclic subgroup of $E[\ell]$ of order $\ell$.
\end{enumerate}
\end{lem}

\begin{proof}
\begin{enumerate}
\item In the ordinary case $\End(E)\otimes\Q$ is an imaginary quadratic field and the action of $\End(E)$ on $E[\ell]\simeq(\Z/\ell\Z)^2$ factors through
$$
\End(E)/\ell\End(E)\hookrightarrow\End(E[\ell])\simeq M_2(\F_\ell).
$$
An  ideal $\p\mid(\ell)$ yields a nontrivial ideal in $\End(E)/\ell\End(E)$ whose annihilator on $E[\ell]$ is a one-dimensional $\F_\ell$-subspace; this annihilator is precisely $E[\p]$, of order $\ell$.

\item In the supersingular case $\End(E)$ is a maximal order in the quaternion algebra $B_{\p,\infty}$. If $\p$ is two-sided with reduced norm $\ell$, standard theory of ideals in maximal orders yields an isogeny $\varphi_{\p}\colon E\to E/\Ker\varphi_{\p}$ whose kernel has cardinality equal to the reduced norm, hence $\ell$. Hence by definition, we have $E[\p]=\Ker(\varphi_{\p})\subset E[\ell]$.

\end{enumerate}

\end{proof}

The set of points whose $x$-coordinates are roots of $\psi_\ell$ is such that we can define a new polynomial :

\begin{defi}
Let $H$ be a subscheme of the form $E[\p]$ as above. The kernel polynomial of $H$ is
\begin{align*}
\psi_H(x)&=\prod_{P\in (H\backslash\{0\}/\{\pm1\})}(x-x(P)).
\end{align*}
which is a factor of  $\psi_\ell$.
\end{defi}

Since $H \subset E[\ell]$ is stable under the action of $\Gal(\overline{k}/k)$, the set of $x$-coordinates 
$\{x(P) \colon P \in H\backslash\{0\}/\{\pm1\}\}$ is stable under $\Gal(\overline{k}/k)$, hence $\psi_H \in k[X]$.
We study the action of $\End(E)$ on $H$.

\begin{lem}\label{stabi}
A subscheme $H$ of $E[\ell]$ is of the form $E[\p]$ for some ideal $\p \subset \End(E)$ if and only if $H$ is stable under the action of $\End(E)$.
\end{lem}

\begin{proof}
Assume first that $H = E[\p]$ for an  ideal $\p \subset \End(E)$.
By definition,
$$
E[\p] = \bigcap_{\phi \in \p} \Ker(\phi).
$$
Let $\gamma \in \End(E)$ and $P \in H$.
For every $\phi \in \p$ we have $\phi(P)=0$, hence
$$
\phi(\gamma(P)) = (\phi \circ \gamma)(P) = 0,
$$
since $\phi\gamma \in \p$.
Thus $\gamma(P) \in \Ker(\phi)$ for all $\phi \in \p$, which means $\gamma(P) \in E[\p] = H$.
Therefore, $H$ is stable under the action of $\End(E)$.

\medskip

Conversely, assume that $H \subset E[\ell]$ is a finite subgroup stable under $\End(E)$.
Define $$
\p_H = \{ \phi \in \End(E) \mid  H\subset\Ker(\phi) \}.
$$
Moreover, by definition $\phi(Q)=0$ for all $Q \in H$ and $\phi \in \p_H$, so $H \subseteq E[\p_H]$.
Conversely, since $H$ is cyclic of order $\ell$ and stable, the ideal $\p_H$ is maximal and annihilates exactly $H$. Hence $E[\p_H] \subseteq H$, and the reverse inclusion is clear, so $E[\p_H] = H$.
\end{proof}

\begin{lem}\label{lem:stability}
Let $H=E[\p] \subseteq E[\ell]$ be a finite étale subgroup and let  $K \subset R_H$ the $k$-subalgebra generated by the coordinates
$\{x(P)\}_{P \in H}$.
If $f \in \End(E)$, then $f^\ast_{|K} \colon K \to K$ is a $k$-algebra automorphism.
\end{lem}

\begin{proof}
Since $ R_H\simeq \prod_{P \in H} k(P)$ and Lemma \ref{stabi}, the map $f^\ast$ permutes the factors via $ P\mapsto f(P)$ and is therefore bijective as a $k$-algebra map. \
In particular, for every idempotent $e_P$ corresponding to the point $P \in H$, we have $f^\ast(e_P) = e_{f(P)}$,
So $f^\ast$ acts as a permutation of primitive idempotents.

This shows that $f^\ast$ is an automorphism, and the restriction $f^\ast_{|K}$ is hence a $k$-automorphism.
\end{proof}

Let $K$ be the splitting field of $\psi_H=\prod_{P\in (H\backslash\{0\}/\{\pm1\})}(x-x(P))$ over $k$. Let $S\subset H$ be a set   of coset representatives for $(H\backslash\{0\}/\{\pm1\})\simeq x(H\backslash\{0\})=x(P_i)_{1 \leq i \leq n}$. We have, for any point $P=P_1\in H$,
$$
S=\{x(P_i)\}_{1\le i\le n}=\{x(iP)\}_{1\le i\le n}=(x_i)_{1\le i\le n}.
$$

Suppose $S$ forms a  basis of $K/k$. Take  an endomorphism $f\in\End(E)$ such that $\deg(f)$ is coprime to $\ell$. Denote by $\bar{f}$ the restriction to $K = \mathbb{P}^1(K) \subset \mathbb{P}^1(\overline{k})$ of the map induced by $f$ on the Kummer curve $E/\{\pm1\}\simeq\mathbb{P}^1$.  We define the $k$-linear map 
$$
\sigma_f\colon\fonction{K}{K}{\sum_{i=1}^{n}\alpha_i
x_i }{\sum_{i=1}^{n} \alpha_i \bar{f}(x_i)}.
$$
We obtain the following commutative diagram of group morphisms :
$$
\begin{tikzcd}
E[\ell] \arrow[r, "f"] \arrow[d, "P \mapsto x(P)"'] & E[\ell] \arrow[d, "P \mapsto 
x(P)"] \\
\mathbb{P}^1 (K) \arrow[r, "\bar{f}"] & \mathbb{P}^1 (K)
\end{tikzcd}
$$

\begin{prop}\label{1}
Let $f\in\End(E)$ be an endomorphism of $E$. The map $\sigma_f \colon K \to K$ induced by the endomorphism $f$  is a field homomorphism.
\end{prop}

\begin{proof}
We know $E/k$ can be viewed as a group scheme of finite type.

Let $H \subset E[\ell]$ be a finite (cyclic) subgroup of order $\ell$. Since $H$ is
finite étale over $k$, its coordinate ring $R_H$ is a finite étale $k$-algebra of rank $\ell$.\
In particular, we have $R_H \simeq \prod_{P \in H} k(P)$, where $k(P)$ denotes the residue field at the closed point $P$.

The maximal ideal $\mathfrak{m}_O$ corresponds to the origin, and since $f(O) = O$, then $f^\ast$ preserves it. The induced map on the
quotient $R_H / \mathfrak{m}_O \simeq k$ is the identity, and $f^\ast$ restricts to an endomorphism of the subalgebra $K \subset
R_H$ generated by the $x$-coordinates of non-trivial torsion points.

Now, let $f \in \End(E)$ be an endomorphism of the elliptic curve. Lemma \ref{stabi} ensures that $f(H)=H$. Then, with  Lemma~\ref{lem:stability}, $f$ induces a group scheme automorphism
$f_H \colon H\to H$, and the corresponding map $f^\ast \colon R_H \to R_H$ on the coordinate ring acts by precomposition on
regular functions, i.e., for all $\gamma \in R_H$, we have $f^\ast(\gamma) = \gamma \circ f$.

We therefore see that any endomorphism $f \in \mathrm{End}(E)$ preserving $H$ naturally acts on the extension $K/k$ via the
map $\sigma_f$.

By Lemma~\ref{lem:stability}, the restriction $f^\ast_{|K}$ is a $k$-automorphism. Hence $\sigma_f$ is a field homomorphism.
\end{proof}

\begin{prop}\label{2}
Assume $(x_i)_{1 \le i \le n}$ forms a  basis of $K/k$, let $f \in \End(E)$ be an endomorphism defined
over $k$. Then $f$ induces a Galois group element  $\sigma_f$ in $\Gal(K/k)$, given  by
$$
\sigma_f\Big(\sum_{i=1}^{n} \alpha_i x_i\Big) = \sum_{i=1}^{n} \alpha_i \bar{f}(x_i).
$$
\end{prop}
\begin{proof}
By Lemma \ref{stabi}, we have $f(H) = H$, the endomorphism $f$ induces  a
homomorphism $f^\ast$ on the coordinate ring $R_H$, and its restriction $\sigma_f$ to the subalgebra $K$ is therefore a field
automorphism.
\end{proof}

Now, we state the main result in the case where $\ell$ is prime:

\begin{prop}\label{thm:main}
Let $k = \F_{q} $ be a finite field of characteristic $p$ and let $\ell$ be a prime such that $\ell\neq p$.
Let $E/k$ be an elliptic curve, and let $\p$ be an ideal  of $\End(E)$ of  norm $\ell$.
Denote by $H=E[\p]$  the subscheme of $E[\ell]$ defined by
$$
H=E[\p]=\bigcap_{\phi\in\p} \Ker(\phi)\simeq \Z/\ell\Z .
$$
Denote by $\psi$ an irreducible factor of the kernel polynomial $\psi_H$ of $H$ and  the degree of $\psi$ by $n=\deg(\psi)$. Let $K$ be the splitting field of $\psi$ over $k$, and suppose that the roots $(x_i)_{1 \le i \le n}$ of $\psi$ form a basis of $K/k$.
If $f \in \End(E)$ is a non trivial endomorphism defined over $k$ whose degree is coprime to $\ell$, then $f$ induces an automorphism
$$
\sigma_f \in \Gal(K/k), \quad \text{defined by} \quad \sigma_f\left( \sum_{i=1}^{n} \alpha_i x_i \right) = \sum_{i=1}^{n} \alpha_i \bar{f}(x_i),
$$
for any $\alpha_i \in K$, where $\bar{f}$ is the induced endomorphism of the Kummer curve $\mathbb{P}^1(k) = E/\{\pm1\}$.
\end{prop}

\begin{proof}
We assume the conditions of Proposition \ref{thm:main}.

Now, we choose a nonzero point $Q \in H \subseteq E[\ell]$ of prime order
$\ell$, with $\ell \neq \mathrm{char}(k)$, which exists due to the cyclicity of $H\subset E[\ell]$.

Then, we fix an endomorphism $f \in \End(E)$, and since $f$ is nonzero, the induced map $\sigma_f$ sends $x(Q)\mapsto x(f(Q))$, and by Proposition~\ref{1}, it defines a field homomorphism of $K = k(x(P))$.
Proposition~\ref{2} ensures that $\sigma_f$ is nontrivial and bijective. Hence, $\sigma_f \in \Gal(K/k)$, which proves Proposition~\ref{thm:main}.
\end{proof}


\section{Construction in the general case}\label{General case}

\subsection{Extension to composite degrees}\label{sec:composite}

The statement of Proposition~\ref{thm:main} does not rely on the primality of the integer $\ell$, $\psi$ denotes a irreducible factor of $\psi_H$.
Indeed, its proof only uses the following ingredients:
\begin{itemize}
\item the existence of a subgroup scheme $H=E[\p]\subset E[\ell]$;
\item the fact that the roots $(x_i)_{1\le i\le n}$ of $\psi$ form a basis of the $k$-vector space $K$;
\item the existence of an endomorphism $f\in\End(E)$ such that $f(H)=H$.
\end{itemize}
Lemma~\ref{lem:stability} and Proposition~\ref{1} ensure that $f$ induces a $k$-morphism $\sigma_f$ on $K$, and
Proposition~\ref{2} implies that $(x_i)$ being a basis yields the explicit expression
$$
\sigma_f\Big(\sum_{i=1}^{n}\alpha_i x_i\Big)=\sum_{i=1}^{n}\alpha_i \bar{f}(x_i),\qquad \forall 1\leq i \leq n, \qquad \alpha_i\in k.
$$

\begin{prop}\label{cor:composite}
The conclusion of Proposition~\ref{thm:main} remains valid when
\begin{itemize}
\item the integer $\ell$ is a product of  distinct primes.
\item the integer $\ell$ is a power of a prime.
\end{itemize}
\end{prop}

\begin{proof}
We proceed by induction on $\omega(\ell)$, the number of prime factors  of $\ell$.

\medskip\noindent\textbf{Base case.} If $\omega(\ell)=1$, i.e. $\ell$ is prime, the statement is exactly Proposition~\ref{thm:main}.

\medskip\noindent\textbf{Coprime decomposition.} Assume the claim holds for all integers with fewer than $\omega(\ell)$ prime factors. We make the case $\omega(\ell)=2$.

Let $\ell_1,\ell_2$ be two different primes such that $\ell=\ell_1\ell_2$ with $\gcd(\ell,p)=1$. Write $H\subseteq E[\ell]$ the subgroup (or subscheme) whose kernel polynomial is considered and denote by $\psi_{H}$ its kernel polynomial. Let $K$ be the splitting field of $\psi_H $ over $k$.

Over an algebraic closure, we have the canonical isomorphism of group schemes
$$
E[\ell]\simeq E[\ell_1] \times E[\ell_2],
$$
and therefore the subgroup $H$ decomposes (after possibly choosing suitable representatives) as a product $H\simeq H_1\times H_2$ with $H_i\subset E[\ell_i]$. Let $\psi_{H_i}$ be the kernel polynomial of $H_i$ and let $K_i$ be the splitting field of $\psi_{H_i}$; by construction the set of roots of $\psi_H $ is the Cartesian product of the sets of roots of $\psi_{H_1}$ and $\psi_{H_2}$, hence
$$
K = K_1K_2 \subset \overline k,
$$
the compositum of $K_1$ and $K_2$. Since  the degrees divide powers of $\ell_1$ and $\ell_2$ respectively, we have 
$$
\gcd\large([K_1:k],[K_2:k]\large)=1
$$
then the extensions $K_1/k$ and $K_2/k$ are linearly disjoint; consequently
$$
[K:k]=[K_1:k][K_2:k]
$$
and $K\simeq K_1\otimes_k K_2$ is a field (the tensor product is a domain by linear disjointness).

By the induction hypothesis, for any endomorphism $f\in\End_k(E)$ which preserves $H_1$ (resp.\ $H_2$) and acts bijectively on it, there is an induced automorphism $\sigma_{f}^{(1)}\in\Gal(K_1/k)$ (resp.\ $\sigma_{f}^{(2)}\in\Gal(K_2/k)$) given by the action on the corresponding roots. If $f$ preserves $H\simeq H_1\times H_2$ (and acts componentwise), then its action on the roots is diagonal: it sends a pair of roots $(x_{1},x_{2})$ to $(\sigma_{f}^{(1)}(x_{1}),\sigma_{f}^{(2)}(x_{2}))$. By the universal property of the compositum, these two actions glue to a unique automorphism $\sigma_f\in\Gal(K/k)$ whose restriction to $K_i$ is $\sigma_{f}^{(i)}$. Thus, the statement holds for $\ell=\ell_1\ell_2$.

By direct induction, we have the result for a finite product of distinct primes.

\medskip\noindent\textbf{Prime powers.}
It remains to handle the case $\ell = z^r$ for a prime $z\neq \Char(k)$ and an integer $r \ge 1$.
Consider the natural filtration of the $\ell$-primary part:
$$
{0}=H[1]\subset H[z]\subset H[z^2]\subset\cdots\subset H[z^r]=H,
$$
where $H[z^i]$ denotes the subgroup of points of order dividing $z^i$.
Let $\psi_i$ be the kernel polynomial of $H[z^i]$ and $K_i$ its splitting field over $k$.
Then we obtain a tower of finite extensions
$$
k = K_0 \subset K_1 \subset \cdots \subset K_r = K,
$$
with $[K_i : K_{i-1}]$ dividing the degree of the kernel polynomial $\psi_i$
at each step.

We proceed by induction on $i$. For $i=1$, the claim follows from Proposition~\ref{thm:main} applied to $H[z]$.

Assume that $\sigma_f$ has been defined on $K_{i-1}$.
Since $f(H)=H$ and $f$ is a rational map over $k$, we have $f(H[z^i]) = H[z^i]$, hence $f$ permutes the abscissæ of $H[z^i]$, i.e the roots of $\psi_i$.
As $\psi_i$ is separable when $\ell \neq \Char(k)$, this permutation extends uniquely to an automorphism of the splitting field $K_i$ fixing $K_{i-1}$.
By construction, this automorphism restricts to the previously defined $\sigma_f$ on $K_{i-1}$.
Induction on $i$ yields an automorphism $\sigma_f \in \mathrm{Gal}(K/k)$ satisfying
$$
\sigma_f(x(P)) = x(f(P)), \quad \forall P \in H.
$$

Combining this prime power case with the coprime decomposition of $\ell$,
the claim follows for all integers $\ell \ge 1$.
\end{proof}

\subsection{Main result}

To conclude this section, we can establish Theorem~\ref{thm::final} :

\begin{proof}[Proof of Theorem~1.1]
To establish this result, we just need to combine Proposition~\ref{thm:main} and Proposition~\ref{cor:composite}.
\end{proof}
We make the following remarks about this result. First, we can see the geometric aspect of our construction :

\begin{rem}
We have constructed $\sigma_f$, which is a geometric realization of the Galois action induced by $f$ on $E[\ell]$.
\end{rem}

Furthermore, we can relate our construction to the Galois representation  theory. Recall, we have the representation
$\rho_{E,\ell}\colon\Gal(\overline{k}/k)\to\Aut(E[\ell])\simeq\GL_2(\mathbb{F}_\ell)$. It describes the action of $\Gal(\overline{k}/k)$ on $E[\ell]$.
\begin{rem}
The induced automorphism
$$
\sigma_f\colon x(P)\in K=k[x(P)\mid P\in H]\mapsto x(f(P))\in K
$$
gives an explicit realization of this action, at level of coordinates field $K$ of points of $H$.

\end{rem}

Finally, we can extend the previous construction to an infinite family of finite field extensions.

\begin{rem}
Consider a family of subgroups $(H_i = E[\p_i])_{i\ge 1}$, where $\p_i$ are prime ideals of degree $1$ in $\End(E)$, and such that the corresponding kernel polynomials $\psi_{H_i}$ are irreducible. Setting $K_i=k[X]/(\psi_{H_i}(X))$, we obtain a family of extensions $(K_i/k)_{i\ge 1}$ whose degrees $[K_i:k]$ are unbounded.
\end{rem}

\subsection{Constructing the Galois group}\label{xxx}
Our construction associates to any endomorphism $f \in \End(E)$ preserving $H$ an automorphism $\sigma_f \in \Gal(K/k)$. When
$H = E[\p]$ for an ideal $\p \subset \End(E)$, the action of $\End(E)$ on $H$ factors through $\End(E)/\p$, and the group of
units $(\End(E)/\p)^*$ embeds naturally into $\Aut(H)$.

Since $\sigma_f$ depends only on the restriction of $f$ to $H$, the construction factors through $(\End(E)/\p)^*$. This yields a canonical group homomorphism
$$
\sigma\colon\fonction{(\End(E)/\p)^*}{\Gal(K/k)}{f}{\sigma_f}.
$$

\begin{rem}
By construction, we have $\{\pm1\}\subseteq\Ker(\sigma)$. Indeed, for all $u \in (\End(E)/\p)^*$ and all points $Q\in H$, we have $x(uQ)=x(Q)$ if and only if $uQ=\pm Q$, so $\Ker(\sigma)\simeq\{\pm1\}$.
\end{rem}

We recall the hypotheses of Theorem~\ref{thm::final}, let $k=\F_q$ be a finite field of characteristic $p$, and let $\ell$ be a positive integer coprime to $p$. Let $E/k$ be an elliptic curve and let $\p\subset \End(E)$ be an ideal of norm $\ell$. 

Set $H=E[\p]=\bigcap_{\phi\in\p}\Ker(\phi)\subseteq E[\ell].$ Let $\psi$ be an irreducible factor of the kernel polynomial $\psi_H$,  and let $K$ be the splitting field of $\psi$ over $k$. 
We assume that the roots $(x_i)_{1\le i\le n}$ form a $k$-basis of $K$ and denotes by $\mathrm{Frob}:(x,y)\mapsto(x^q,y^q)$ the geometric Frobenius.

\begin{cor}\label{isomorphisme}
Under the conditions of Theorem \ref{thm::final}, 
\begin{enumerate}
    \item\label{surjective} The group homomorphism
    $$
    \sigma\colon\fonction{(\End(E)/\p)^*}{\Gal(K/k)}{f}{\sigma_f}.
    $$
    is a surjective.
    \item If $\deg(\psi)=\frac{\varphi(\ell)}{2}$, then, 
    $$
    (\Z/\ell\Z)^*/\{\pm1\}\simeq\Gal(K/k)\simeq(\End(E)/\p)^*/\{\pm1\}. 
    $$
    where $\varphi$  is Euler’s totient function.
\end{enumerate}

\end{cor}

\begin{proof}
\begin{enumerate}
    \item The Frobenius automorphism $\pi_q$ generates the Galois group $\Gal(K/k)$. The image under $\sigma$ of the geometric Frobenius $\mathrm{Frob}$ is the Frobenius automorphism $\pi_q$, which generates $\Gal(K/k)$. Hence $\sigma$ is surjective.
    \item On one hand, we have $\Gal(K/k)\simeq(\End(E)/\p)^*/\Ker(\sigma)$. On the other, as $H$ is cyclic of order $\ell$, we have  $(\End(E)/\p)^*\simeq(\Z/\ell\Z)^*$ and by the hypothesis $|\Gal(K/k)|=\deg(\psi)$. We conclude 
    $$
    (\Z/\ell\Z)^*/\{\pm1\}\simeq\Gal(K/k)\simeq(\End(E)/\p)^*/\{\pm1\}.
    $$
\end{enumerate}
\end{proof}

\begin{rem}
The condition $\deg(\psi)=\frac{\varphi(\ell)}{2}$ is equivalent to the equality of polynomials $\psi_H=\psi$. Indeed, we have $E[\ell]\simeq(\Z/\ell\Z)^2$, $H=E[\p]$ where $\p$ is an ideal of $\End(E)$ of norm $\ell$. As, for all points $P\in H\backslash\{0\}$, $x(P)=x(-P)$, then $\deg(\psi_H)=\frac{\varphi(\ell)}{2}$. By equality of degree, we can conclude $\psi=\psi_H$.
\end{rem}

\begin{rem}
For all $1\leq i\leq n$, we have $\sigma_f(x_i)\in\{x_j\mid1\leq j \leq n\}$, so $\sigma_f$ induces a permutation of the roots of $\psi_H$. Also, we have group isomorphisms
$$
\Gal(K/k)\simeq(\End(E)/\p)^*/\{\pm1\}\simeq(\Z/\ell\Z)^*/\{\pm1\}.
$$
Identifying each index $i$ with the corresponding element of $(\Z/\ell\Z)^*/\{\pm1\}$ via these isomorphisms, $\sigma_f$ acts as multiplication by the class of $f$, i.e.\ as the permutation $i \mapsto f\cdot i$ in $(\Z/\ell\Z)^*/\{\pm1\}$.
\end{rem}

The condition $\psi=\psi_H$ ensures that the kernel polynomial is irreducible. By the Chebotarev density theorem applied to the function field of the curve, the Frobenius element $\pi_q$ is distributed as a random element within the Galois group $\Gal(K/k)\subset(\Z/\ell\Z)^*/\{\pm1\}$. 

\begin{rem}
Heuristically, this means that for a random choice of the curve and the prime $\ell$, the probability that $\psi_H$ is irreducible is $\varphi(n)/n$.
\end{rem}

The polynomial $\psi_H$ is irreducible if and only if its roots form a single 
orbit under the action of $\pi_q$.

Consequently, $\psi=\psi_H$ holds if and only if the Frobenius $\pi_q$ is a primitive element of the group 
$(\Z/\ell\Z)^*/\{\pm1\}$. This happens with probability $\frac{\varphi(n)}{n}$ where $n=\frac{\varphi(\ell)}{2}$.


\section{Algorithmic considerations}\label{sec:algo}

The purpose of this section is not to introduce an optimized algorithmic variant of the construction, but to show that the geometric construction of Section \ref{sec:construction} admits a direct and explicit computational realization.

\subsection{Constructing an endomorphism of elliptic curves from isogeny cycles}
We want to determine an explicit endomorphism $f \in \End(E)$ of degree $d$ coprime to $\ell$.

\medskip

Our goal is to construct an endomorphism $f\in\End(E)$ that can be evaluated efficiently. To do this, we choose to construct it in the form of a cycle of isogenies of small degrees  such that there exist elliptic curves $(E_i)_{1\leq i\leq n}$ and isogenies $(\phi_i)_{1\leq i\leq n}$ such that $\phi_i \colon E_i \to E_{i+1}$, defining a path in an isogeny graph satisfying $E_1 = E_n = E$, and for the endomorphism  $f = \phi_n \circ \cdots \circ \phi_1 \colon E \to E$.  

By construction, a closed (oriented) isogeny cycle defines an endomorphism $ f \in \mathrm{End}(E) $ which preserves the chosen cyclic subgroup $ H \subset E[\ell] $.

\subsubsection{Ordinary case}
Such a cycle can be constructed by using the structure of ordinary isogeny graphs; see Chapter~III in \cite{Kohel1996} for further details. Another good reference for calculating endomorphism rings in the ordinary case can be found in \cite{bissonsutherland}.
\medskip

Let $\OO$ be an order in a quadratic field, let $\pi \in \OO$ be an element of norm $q=p^r$, and let $\mathcal{E}ll(\OO,\pi)=\{E_i\}_{1\leq i\leq h}$ be a set of representatives of elliptic curves over $k$ equipped with isomorphisms $\eta_i\colon\OO\to\End(E_i)$ sending $\pi$ to the Frobenius endomorphism $\pi_i\in\End(E_i)$. Let $\ell$ be a rational prime and let $\h\subset\OO$ be a prime ideal above $\ell$. Via the isomorphism $\eta_i$, the ideal $\h$ defines a finite cyclic subgroup scheme $G_i = E_i[\h] \subset E_i[\ell]$.

Let $S=\{\h_1,\dots,\h_t\}$ be a set of prime ideals of $\OO$ such that $\h\notin S$, and let $\psi_i$ be an irreducible factor of the kernel polynomial of $G_i$. We define an isogeny graph $\Gamma_S(\OO,\pi)$ whose vertices are pairs $(E_i,\psi_i)$ with $E_i\in\mathcal{E}ll(\OO,\pi)$ and $\psi_i$ an irreducible factor of the kernel polynomial of $G_i=E_i[\h]$, and whose edges correspond to isogenies $E_i\to E_j$ with kernel $E_i[\h_k]$ for some $\h_k\in S$.

We view the set of roots $V(\psi_i)=\Spec(\F_q[X]/(\psi_i))$ as a finite étale subscheme of the Kummer quotient $G_i\setminus\{0\}/\{\pm1\}$, whose closed points correspond to Galois orbits of the $x$-coordinates of non-zero points of $G_i$.

\subsubsection{Supersingular case}
In the case of supersingular elliptic curves, see \cite{colo} for more details on the theory of oriented isogeny graphs.

In the supersingular case, the isogeny graph admits an orientation. Its vertices are oriented pairs $(E_i,Z_i)$, where $Z_i\subset E_i[\ell]$ is a cyclic subgroup, and its edges are isogenies $(E_i,Z_i)\to(E_j,Z_j)$ preserving the orientation.

A closed path in this oriented graph determines a finite subgroup $Z\subset E[\ell]$ which is preserved by the induced endomorphism $f\in\End(E)$. The associated coordinate subalgebra $K = k[Z]\subset \OO_{E[\ell]}$ is therefore stable under the action of $f$, and the induced action defines an automorphism $\sigma_f\in\Gal(K/k)$.

\subsection{Algorithm}
\begin{alg}\label{alg:sigmaf}

\textbf{Input :}
\begin{itemize}
\item A finite field $k=\mathbb{F}_q$ , and an elliptic curve $E/k$.
\item An integer $\ell$ such that $\Char(k)$ does not divides $\ell$, $\ell\neq2$, $H=E[\p]\subset E[\ell]$ and  assume its division polynomial $\psi_H $ is irreducible with  $\deg(\psi_H)=n$.
\item An endomorphism $f\in\End_k(E)$ with $\deg(f)$ coprime to $\ell$.
\end{itemize}

\textbf{Output :}
The  automorphism $\sigma_f\in\Gal(K/k)$ defined by

\begin{align*}
\sigma_f\left( \sum_{i=1}^{n} \alpha_i x_i \right)&= \sum_{i=1}^{n} \alpha_i \bar{f}(x_i).
\end{align*}

\textbf{Procedure.}

\begin{enumerate}
\item Let $K=k[x]/(\psi_H)$ and let $\mathcal B=(x_i)_{1\le i\le n}$ be  the roots of $\psi_H$ which form a basis of $K/k$.
\item For each basis element $x_i$ compute its image $\bar{f}(x_i)$ under $f$. By abuse of notation, $f(x(P))$ denotes the x-coordinate of the point $f(P)$ where $P$ is a point with x-coordinate $x_i$.
\item Extend $k$-linearly: for any $z=\sum_{i=1}^{n}\alpha_i x_i\in K$, define
$$
\sigma_f(z)=\sum_{i=1}^n\alpha_i\bar{f}(x_i).
$$
This map $\sigma_f:K\to K$ is a $k$-automorphism; if needed, represent it by the $n\times n$ matrix whose $i$-th column contains the coordinates of $f(x_i)$ in the basis $\mathcal B$.
\end{enumerate}
\end{alg}

\subsection{Complexity Analysis}

Throughout this section, complexity estimates are given in terms of bit operations ($\tilde{O}$ hides polylogarithmic 
factors). Now, we consider the endomorphism $f$ like an input. 

\subsubsection{Computing automorphisms using Frobenius}

The classical way to compute an automorphism is as follows.

The automorphism group $\Gal(K/k)$ is generated by the arithmetic Frobenius $\pi_q$ whose action on $K=k[x]/(\psi_H(x))$, 
with $[K:k]=n$, is obtained by reducing $x^q$ modulo $\psi_H(x)$. Using fast exponentiation, evaluating $\pi_q$ requires
$$
\tilde{O}(n(\log q)^2),
$$
because representing an element of $K$ costs $O(n\log q)$;  multiplication in $K$ costs $\tilde{O}(n\log q)$ and  fast
exponentiation costs $\tilde{O}(n(\log q)^2)$.

\begin{rem}
Multiplication in certain finite fields can be optimized by minimizing the bilinear complexity \cite{BalletChaumine2004}.
\end{rem}

The automorphism $\pi_q$ thus obtained is the distinguished generator of the cyclic group $\Gal(K/k)$.
Other elements of $\Gal(K/k)$ can be defined as powers of Frobenius, $\pi_q^i$, $0 \le i < n$, whose evaluation costs
$$
\tilde{O}(i n (\log q)^2)
$$
additional modular compositions in $K$.

\subsubsection{Computing automorphisms using endomorphisms}

\paragraph{Evaluation of isogeny}
Let $C_{\mathrm{eval}}(d)$ denote the complexity of the evaluation of an isogeny $f$ of degree $d$.

For a separable $d$-isogeny, Vélu’s formulas gives $C_{\mathrm{eval}}(d) = \tilde{O}(d^2\log q)$.

For scalar multiplication $[d]$, the decomposition $[d] = \widehat{\varphi} \varphi$ from \cite{Kohel2005} reduces the cost to $\tilde{O}(d\log q)$.

For endomorphisms given as compositions of prime degree $p_i$, with \cite{Bernstein2020} we have
$$
C_{\mathrm{eval}}(p_i) = \tilde{O}(\sqrt p_i),
$$
When $d = \prod_i p_i^{e_i}$ is a composite integer, we need to decompose the isogeny $f$; after that, we have a cost
$$
\tilde{O}(\sum_i e_i \sqrt{p_i} \log q).
$$

This applies only to the initial evaluation of $f(x_i)$ for $i=1,\dots,n$ during precomputation.

\paragraph{Analysis of Algorithm}

We now analyze Algorithm~\ref{alg:sigmaf}. Let $n=\deg(\psi_H)$, constructing the field $K$ and a $k$-basis costs $O(n \log q)$
; computing the images $f(x_i)$ for the $n$ basis elements costs $\tilde{O}(n  C_{\mathrm{eval}}(d))$ ; expressing these
images in the chosen basis and building the transition matrix costs $\tilde{O}(n^3 \log q)$.

The third step corresponds to a one-time precomputation: once the matrix representing the action of $f$ on the basis is known,
the automorphism $\sigma_f$ can be evaluated efficiently for all subsequent elements of $K$. This precomputation cost
$O(n  C_{\mathrm{eval}}(d)) + \tilde{O}(n^3 \log q)$.

\begin{rem}
If the roots of $\psi_H$ form a normal basis of $K/k$, the Frobenius action is cyclic and stable under Galois conjugation.
In this case, expressing $\bar{f}(x_i)$ in the basis reduces to linear-time operations, and the precomputation cost drops to
$\tilde{O}(n\log(\ell)^2)$.
\end{rem}

To conclude, evaluation of $\sigma_f$ on a single element of $K$ cost $\tilde{O}(n^2)$.

\subsubsection{Comparison}

The main interest of the endomorphism-based construction is not asymptotic speed, but the ability to generate explicit Galois automorphisms without revealing their discrete logarithm with respect to Frobenius.

The Frobenius-based method is optimal when we want to compute all automorphisms.

In contrast, the endomorphism–isogeny approach produces many distinct automorphisms directly from elements of $\End(E)$, without using Frobenius.

From a complexity perspective, the method becomes interesting if we denote the degree of the isogeny $f$ by  $d = \prod_i p_i^{e_i}$, when $\log q \gg n  \sum_i e_i \sqrt{p_i}.$


\section{Examples}\label{sec:applications}

We recall that if $E$ is  ordinary (resp. supersingular), then $\End(E)$ is abelian (resp. non-abelian). We want to show that the type of the elliptic curves is not important for the construction.

\subsection[Ordinary case]{Ordinary case\footnote{\textsc{Magma} implementations are available at \url{https://gaati.org/djambae/pdf/Recherche/ordinaire.m}}}
Consider the finite field $k=\F_7$, define  an elliptic curve $E/k : y^2=x^3+2x+4$. Let $\ell=11$. The ring $\End(E)$ is an order in a quadratic field and we have $E[11]\simeq(\Z/11\Z)^2$. We can choose
$$
\psi_H(X) = X^5+2X^4+5X^3+3X^2+2X+2
$$
for the division polynomial of $H\simeq\Z/11\Z$. This polynomial has degree 5 and is irreducible on $\F_7[X]$. We thus have
$$
K=k[X]/(\psi_H(X))\simeq\F_{7^5}.
$$
Denote by $(x_i)_{1\leq i\leq5}$ the roots of $\psi_H$, take a nonzero point  $P\in H\backslash\{0\}$, take a root of $\psi_H$ denoted by $x_1=x(P)$ and take an endomorphism $f=[3]\in\End(E)$, since $f$ acts like a permutation of points of $H$, so we have
$$
\bar{f}=(x_1x_3x_2x_5x_4),
$$
which is a 5-cycle (of order 5). To conclude, we have the explicit expression
$$
\forall x=\sum_{i=1}^5 \alpha_i x_i\in K, \quad \sigma_f(x)=\alpha_4 x_1+\alpha_3 x_2+\alpha_1 x_3+\alpha_5 x_4+\alpha_2 x_5.
$$

\begin{rem}
In this case, we have $\sigma_f=\pi_7^2.$
\end{rem}

\subsection[Supersingular case]{Supersingular case\footnote{\textsc{Magma} implementations are available at \url{https://gaati.org/djambae/pdf/Recherche/supersingulier.m}}}

Let $k=\F_5$ be a finite field and let $E/k : y^2+y=x^3$ be a supersingular elliptic curve. Let $\ell=7$. Then $\End(E)$ is an order in a quaternion algebra.

A choice for the division polynomial of $H\simeq \Z/7\Z$ is
\[
\psi_H(X)=X^3+X^2+X+3.
\]
In this case, the order of the Galois group $\Gal(K/k)$ is $3$ since $\psi_H$ is irreducible. With this data, we have
\[
K=k[X]/(\psi_H(X))\simeq \F_{5^3}.
\]
Denote by $(x_i)_{1\leq i\leq 3}$ the roots of $\psi_H$, take a nonzero point $P\in H\setminus\{0\}$, take a root of $\psi_H$ denoted by $x_1=x(P)$, and take the endomorphism $f=[2]\in\End(E)$. We can describe the action of $f$ on $H$ as the permutation
\[
\bar{f}=(x_1\;x_2\;x_3).
\]

To conclude, as in the ordinary case, we can determine an explicit expression of the induced automorphism:
\[
\forall x=\sum_{i=1}^3 a_i x_i\in K,\quad 
\sigma_f(x)=a_1 x_2+a_2 x_3+a_3 x_1.
\]

\begin{rem}
In this case, we have $\sigma_f=\pi_5^2.$
\end{rem}


\printbibliography

\end{document}